%% file: Version_5.tex
\DeclareSymbolFont{rsfs}{U}{rsfs}{m}{n}

\documentclass[10pt,a4paper,oneside,leqno,noamsfonts]{amsart}
           \makeatletter
           \newcommand{\mylabel}[2]{#2\def\@currentlabel{#2}\label{#1}}
           \renewcommand\@biblabel[1]{#1.}
           \makeatother
\usepackage[english]{babel}
\usepackage[dvipsnames]{xcolor}
\usepackage{graphicx,pifont}
\usepackage[a4paper,top=2.8cm,bottom=2.8cm,left=2.7cm,
           right=2.7cm,bindingoffset=5mm,marginparwidth=60pt]{geometry}
\usepackage[utf8]{inputenc}
\usepackage{braket,caption,comment,mathtools,stmaryrd}
\usepackage[usestackEOL]{stackengine}
\usepackage{multirow,booktabs,microtype,relsize}
\usepackage[colorlinks,bookmarks]{hyperref} 
      \hypersetup{colorlinks,%
            citecolor=britishracinggreen,%
            filecolor=black,%
            linkcolor=cobalt,%
            urlcolor=black}
      \numberwithin{equation}{section}
\input{macros.tex}

\allowdisplaybreaks 
\usepackage{dirtytalk}
\usepackage{csquotes} 

\title[Canonical vertex formalism in DT theory of toric CY  4-folds]{Canonical vertex formalism in DT theory of toric Calabi-Yau 4-folds}

\author{Sergej Monavari}
\address{Mathematical Institute, Utrecht University, P.O.~Box 80010 3508 TA Utrecht, The Netherlands}
\email{s.monavari@uu.nl}

\begin{document}
\maketitle

\begin{abstract}
Motivated by previous computations of Y. Cao, M. Kool and the author, we propose  square roots and sign rules for the vertex and edge terms that compute Donaldson-Thomas invariants of a toric Calabi-Yau 4-fold, and  prove that they are canonical, exploiting the combinatorics of plane and solid partitions. 
\end{abstract}

\section{Introduction}
\subsection{DT theory of Calabi-Yau 4-folds}
Donaldson-Thomas invariants were firstly defined in Thomas' thesis \cite{Tho_Casson_inv} as a way to count stable sheaves on  projective Calabi-Yau 3-folds\footnote{Here we denote by \emph{Calabi-Yau variety} $X$  a smooth complex quasi-projective variety with $K_X\cong \oO_X$ and $b_1(X)=0$.}, and were recently extended to Calabi-Yau 4-folds in the seminal work of Cao-Leung \cite{CL_DT4}.\\
Let $X$ be a quasi-projective variety,  $\Hilb^n(X,\beta)$ the Hilbert scheme of closed subschemes $Z\subset X$ with proper support in  homology class $[Z]=\beta\in H_2(X, \BZ)$ and $\chi(\oO_Z)=n$ and denote by $\CI$ the universal ideal sheaf on $X\times  \Hilb^n(X,\beta)$.  By the work of   Huybrechts-Thomas \cite{HT_obstruction_theory},  the Atiyah class gives an \emph{obstruction theory} on $\Hilb^n(X,\beta)$
\begin{equation}\label{eqn: obstruction theory}
     \RR\hom_\pi(\CI, \CI)^\vee_0[-1]\to \BL_{\Hilb^n(X,\beta)}, 
\end{equation}
where 
$(\cdot)_0$ denotes the trace-free part,
$\RR\hom_\pi=\RR\pi_*\circ \RR\hom$,   $\pi:X\times \Hilb^n(X,\beta)\to \Hilb^n(X,\beta)$ and $  \BL_{\Hilb^n(X,\beta)}$ is the truncated cotangent complex.\\
If $X$ is a smooth projective 3-fold, the obstruction theory (\ref{eqn: obstruction theory}) is \emph{perfect}, so $M$ carries a   virtual fundamental class; Donaldson-Thomas invariants are defined by integrating insertions against it. Unfortunately, if $X$ is a smooth projective 4-fold, the obstruction theory fails to be perfect, so the machineries of Behrend-Fantechi \cite{BF_normal_cone} and Li-Tian \cite{LT_virtual_cycle} do not produce a virtual class. Nonetheless, if $X$ is a projective Calabi-Yau 4-fold, Borisov-Joyce~\cite{BJ_virtual_fundamental_classes} 
constructed a virtual fundamental class
\[
[\Hilb^n(X,\beta)]^{\vir}_{o(\CL)} \in H_{2n}(\Hilb^n(X,\beta), \mathbb{Z})
\]
which depends on a choice of orientation of $\Hilb^n(X,\beta)$, i.e. a choice of  square root of the isomorphism
\[
Q: \CL\otimes \CL \xrightarrow{\sim} \mathcal{O}_{\Hilb^n(X,\beta)}, 
  \]
induced by Serre duality pairing, where $ \CL:=\mathrm{det}(\RR \hom_{\pi}(\CI, \CI))$ is the determinant line bundle. The existence of orientations was  proved 
 for arbitrary compact Calabi-Yau 4-folds by Cao-Gross-Joyce \cite{CGJ_orientability} and 
 in the non-compact setting by Bojko \cite{Bojko_orientations}.
  Recently Oh-Thomas \cite{OT_1} proposed an alternative (algebraic) construction of the virtual cycle
\[[\Hilb^n(X,\beta)]^{ \vir}_{o(\CL)}\in A_n\left(\Hilb^n(X,\beta), \BZ\left[\frac{1}{2}\right]\right),\]
which  coincides with Borisov-Joyce virtual cycle under the cycle map and proved a virtual localization formula. Morally, the virtual class is obtained (at least locally) as the zero locus of an \emph{isotropic} section of an $SO(r, \BC)$-bundle   over a smooth ambient space.
\subsection{Hilbert schemes of toric CY 4-folds}
It is in general very difficult to compute Donaldson-Thomas type invariants of a Calabi-Yau 4-fold. Here we assume that  $X$ is a  toric Calabi-Yau 4-fold and denote by 
   $\TT=\set{t_1t_2t_3t_4=1}\subset (\BC^*)^4$  the subtorus preserving the Calabi-Yau form of $X$. The $\TT$-action naturally lifts to $ \Hilb^n(X,\beta)$, making  the obstruction theory (\ref{eqn: obstruction theory})   $\TT$-equivariant  by \cite[Thm. 53]{Ric_equivariant_Atiyah}. Despite  being almost never the  Hilbert scheme $ \Hilb^n(X,\beta)$   proper, its fixed locus  $ \Hilb^n(X,\beta)^\TT$   is reduced, 0-dimensional (Lemma \ref{lemma: finitely many reduced points}) and the induced obstruction theory is trivial (Prop. \ref{prop: induced obstruction is trivial}), so we can \emph{define} invariants $\TT$-equivariantly by means of Oh-Thomas virtual localization theorem (cf. \cite[Thm. 7.1, Rem. 7.4]{OT_1}).
   \begin{definition}\label{def: DT invariants}
   Let $\gamma\in H^*_\TT(\Hilb^n(X,\beta))$. The $\TT$-equivariant Donaldson-Thomas invariants of $X$ are 
   \begin{equation}\label{eqn: DT_inv}
        \DT_n(X,\beta;\gamma)=\sum_{Z\in \Hilb^n(X,\beta)^\TT}\sqrt{e^\TT}(-T^{\vir}_Z)\cdot \gamma|_Z\in \frac{\BQ(\lambda_1, \lambda_2, \lambda_3, \lambda_4)}{(\lambda_1+\lambda_2+\lambda_3+\lambda_4)},
   \end{equation}
   where $T^{\vir}\in K^0_\TT(\Hilb^n(X,\beta))$ is the class of the dual of the obstruction theory (\ref{eqn: obstruction theory}) and $\lambda_i=c_1^\TT(t_i)$.
   \end{definition}
   These invariants had been intensively studied (e.g. \cite{CK_DT-PT, CK1, CT_GV_via_descendant, CKM_K_theoretic}), where $\gamma$ is respectively a primary, descendant, tautological or $K$-theoretic insertion.\\
Here we denoted by $ \sqrt{e^\TT}(\cdot)$  the ($\TT$-equivariant) square root Euler class defined by Edidin-Graham in \cite{EG_char_classes_quadratic_bundles} (cf.  also \cite[Sec. 7]{OT_1}). The definition of this class is far  from being explicit since relies on the chosen orientation of  $\Hilb^n(X,\beta)$. To ease our life, we define the \emph{square root} of a $\TT$-representation.
\begin{definition}
Let $V\in K^0_\TT(\pt)$ be a virtual $\TT$-representation. We say that $T\in K^0_\TT(\pt)$ is a \emph{square root} of $V$ if
\[
V=T+\overline{T}\in K_\TT^0(\pt),
\]
where $\overline{(\cdot)}$ denotes the dual $\TT$-representation.
\end{definition}
Let $ V$ admit a   square root  $T$ in $K^0_\TT(\pt)$. Its square root Euler class satisfies
\[\sqrt{e^\TT}(V)=\pm e^\TT(T), \]
where the sign depends on the chosen orientation. Therefore, we can compute the invariants (\ref{eqn: DT_inv}) by explicitly finding a square root of the virtual tangent bundle, at the price of introducing a (non-explicit!) sign at every $\TT$-fixed point of $ \Hilb^n(X,\beta)^\TT$.\\

Finding the correct signs given a square root of the virtual tangent space is one of the major difficulties of the theory, as the number of possible choices of signs rapidly increases with the number of fixed points, which makes the application  of  the virtual localization formula not feasible --- we recall that the virtual localization formula is one of the rare tools at our disposal to perform computations in DT theory.  The main goal of this paper is to propose various square roots of the virtual tangent bundle with \emph{sign rules}, which would make the invariants (\ref{eqn: DT_inv}) effectively computable. As evidence to our proposals, we show that they
\begin{itemize}
    \item are canonical (Theorem \ref{thm: orientation points}, \ref{thm: orientation curves}, \ref{thm: orientation edge}),
    \item play well with the dimensional reduction studied in \cite[Sec. 2.1]{CKM_K_theoretic} (Remark \ref{rem: dim red points}, \ref{rem: dim red curves}, \ref{rem: global sign patching}),
    \item are consistent with previous computations  \cite{CK1, CK_higher_counting, CK_DT-PT, CKM_K_theoretic,Nek_magnificient_4, NP_colors}, both in the math and physics literature.
\end{itemize}
To each  fixed point in $\Hilb^n(X,\beta)^\TT$, we  associate some combinatorial data using \emph{solid partitions}, which can be displayed as arrangements of boxes in $\BZ^4_{\geq 0}$ (cf. Section \ref{sec: toric varieties}). We exploit the combinatorics of solid partitions to propose our  sign rules and to prove they are canonical.
 It is important to notice that  these signs 
 carry independent interesting information on the combinatorics of solid partitions --- a field where not much is currently known (see Section \ref{sec: magn four}).
 
For $X=\BC^4$, our proposals generalize the ones already discussed by  Nekrasov-Piazzalunga \cite[Sec. 2.4]{NP_colors} for $\beta=0$  and by Cao-Kool and the author \cite[Rem. 1.18]{CKM_K_theoretic} when the $\TT$-invariant subscheme $Z\subset X$ is supported in at most two $\TT$-invariant lines. Here we deal with the full theory, where $Z$ can be supported in four $\TT$-invariant lines.
\subsection{Relation to String Theory}
Donaldson-Thomas invariants of Calabi-Yau 4-folds --- and their $K$-theoretic refinement --- appear in String Theory as the result of  supersymmetric localization of $U(1)$ super-Yang–Mills theory with matter on a Calabi–Yau 4-fold, studied by Nekrasov-Piazzalunga \cite{Nek_magnificient_4, NP_colors} and by Bonelli-Fasola-Tanzini-Zenkevich \cite{BFTZ_ADHM8D} with an ADHM-type construction. In these cases, the authors study the quantum mechanics of a system of $D0$-$D8$ branes. Our results would allow to easily perform computations in the presence of points and \emph{curves}, that is for a system of $D0$-$D2$-$D8$ branes.
\subsection{Magnificent four}\label{sec: magn four}
MacMahon \cite{MacMahon_combinatorics} classically solved the combinatorial problem of enumerating
 \emph{partitions} and \emph{plane partitions}. Remarkably, his conjectural formula for \emph{solid partitions} fails and no closed expression is expected to exist. From a modern geometric point of view, the enumeration of partitions and plane partitions appears as a suitable limit of $K$-theoretic  invariants of $\Hilb^{n}(\BC^2)$ and $\Hilb^{n}(\BC^3)$, exploiting respectively the smooth structure and the symmetric perfect obstruction theory --- cf. \cite{Okounk_Lectures_K_theory}. In the case of $\Hilb^{n}(\BC^4)$ the Borisov-Joyce/Oh-Thomas virtual structure crucially depends on a choice of \emph{orientation}, which is responsible for the ambiguity of sign at each fixed point --- a phenomenon that does not arise in lower dimension --- and partially explains the failure of MacMahon's conjecture, where the signs are not playing any rôle. For $d\geq 5$ there are no known virtual structures on $\Hilb^{n}(\BC^d)$ and no direct analogues of the combinatorial formulas in dimensions 2,3,4 seem to hold --- see \cite{CK_higher_counting}. This may indicate that dimension 4 is special and we believe that our sign conjectures could play an important rôle in Combinatorics as well; quoting Nekrasov's \emph{Magnificent four} \cite{Nek_magnificient_4}
 \begin{displayquote}
 \emph{"The adjective 'Magnificent' reflects this author’s conviction that the dimension four is the maximal dimension where the natural albeit complex-valued probability distribution exists."}
 \end{displayquote}
\subsection*{Acknowledgments} 
I am grateful to Martijn Kool  for asking a question that led to this paper and for the many suggestions. I wish to thank Arkadij Bojko and Yalong Cao  for many conversations on orientations in DT theory of Calabi-Yau fourfolds. I also thank the anonymous referee for several helpful comments, which improved the exposition of the paper. S.M.~is supported by NWO grant TOP2.17.004.
\section{Hilbert schemes of toric Calabi-Yau 4-folds}\label{sec: Hilbert scheme toric}
\subsection{Toric varieties}\label{sec: toric varieties}
Let $X$ a toric Calabi-Yau 4-fold and consider the Hilbert scheme $\Hilb^n(X,\beta)$, parametrizing closed subschemes $Z\subset X$ with proper support in the  homology class $[Z]=\beta\in H_2(X, \BZ)$ and $\chi(\oO_Z)=n$. Denote by $\Delta(X)$ the Newton polytope of $X$, by $V(X)$ its set of vertices and by $E(X)$ its set of edges. Vertices $\alpha\in V(X)$ correspond to $ (\BC^*)^4 $-fixed point of $X$, each contained in a maximal $ (\BC^*)^4$-invariant open subset $U_\alpha\subset X$. Edges $\alpha\beta\in E(X)$ correspond to $ (\BC^*)^4 $-invariant lines $L_{\alpha\beta}\cong\BP^1$, whose normal bundle is
\[  
N_{L_{\alpha\beta}/X}\cong \oO(m_{\alpha\beta})\oplus\oO(m_{\alpha\beta}')\oplus\oO(m_{\alpha\beta}'')
\]
and satisfies $m_{\alpha\beta}+m_{\alpha\beta}'+m_{\alpha\beta}''=-2 $, being $X$ a Calabi-Yau variety. We may choose coordinates $t_i$  on $(\BC^*)^4$ and $x_i$ on  $U_\alpha$ such that the $(\BC^*)^4 $-action on $U_\alpha$ is determined by
\[(t_1,t_2,t_3,t_4)\cdot x_i=t_ix_i.    \]
If the line $L_{\alpha\beta}$ is defined in these coordinates by $\set{x_2=x_3=x_4=0}$, the transition function between the charts $U_\alpha$ and $U_\beta$ are of the form 
\[(x_1,x_2,x_3,x_4)\mapsto (x_1^{-1}, x_2x_1^{-m_{\alpha\beta}},x_3x_1^{-m_{\alpha\beta}'},x_4x_1^{-m_{\alpha\beta}''} ). \]
Denote by $\TT=\set{t_1t_2t_3t_4=1}\subset (\BC^*)^4$ the subtorus preserving the Calabi-Yau form of $X$. The $(\BC^*)^4 $-action and the  $\TT$-action naturally lift on $\Hilb^n(X,\beta)$, whose fixed locus is reduced and $0$-dimensional.
\begin{lemma}[{\cite[Lemma 2.1, 2.2]{CK_DT-PT}}]\label{lemma: finitely many reduced points}
We have an isomorphism of schemes
\[\Hilb^n(X,\beta)^\TT=\Hilb^n(X,\beta)^{(\BC^*)^4 },\]
which consists of finitely many reduced points.
\end{lemma}
We recap the description of the $(\BC^*)^4$-fixed locus, which  is completely  analogous to \cite[Sec. 4.2]{MNOP_1} in the setting of toric 3-folds. For an extensive treatment in the case of toric 4-folds, look at   \cite[Sec. 2.1]{CK_DT-PT}.\\

Let $Z\in \Hilb^n(X,\beta)^{(\BC^*)^4 }$ and $I$ be its ideal sheaf; $Z\subset X$ is preserved by the torus action, hence it must be supported on the $(\BC^*)^4 $-fixed points (corresponding to vertices $\alpha\in V(X)$) and $(\BC^*)^4$-invariant lines of $X$ (corresponding to edges $\alpha\beta\in E(X)$). Since $I$ is $(\BC^*)^4$-fixed on each open subset, $I$ must be defined on $U_\alpha$ by a monomial ideal
\[I_\alpha=I|_{U_\alpha}\subset \BC[x_1,x_2,x_3,x_4],\]
and may also  be viewed as a solid partition $\pi_\alpha$,
\begin{equation}\label{eqn: partition to ideal}
    \pi_\alpha=\left\{(k_1,k_2,k_3,k_4), \prod_{i=1}^4x_i^{k_i}\notin I_\alpha\right\}\subset \BZ^4_{\geq 0}.
\end{equation}
The associated subscheme of $I_\alpha$ is (at most) 1-dimensional. The corresponding partition $\pi_\alpha$ may be infinite in the direction of the coordinate axes. If the solid partition is viewed as a box diagram in $\BZ^4$, the vertices in (\ref{eqn: partition to ideal}) are determined by the \emph{interior} corners of the boxes, the corners closest to the origin.

The asymptotics of $\pi_\alpha$ in the coordinate directions are described by four ordinary finite-size plane partitions. In particular, in the direction of the $(\BC^*)^4 $-invariant curve $L_{\alpha\beta}$ given by $\set{x_2=x_3=x_4=0}$, we have the plane partition $\lambda_{\alpha\beta}$ with the following diagram
\begin{align*}
    \lambda_{\alpha\beta}&=\set{(k_2,k_3,k_4):   x_1^{k_1}x_2^{k_2}x_3^{k_3}x_4^{k_4}\notin I_\alpha, \forall k_1\geq 0}\\
    &=\set{(k_2,k_3,k_4): x_2^{k_2}x_3^{k_3}x_4^{k_4}\notin I_{\alpha\beta}}\subset \BZ^3_{\geq 0},
\end{align*}
where 
\[
I_{\alpha\beta}=I|_{U_{\alpha}\cap U_{\beta}}\subset \BC[x_1^{\pm 1}, x_2, x_3, x_4].
\]
The vertices of $\lambda_{\alpha\beta}$ defined above are the interior corners of the squares of the associated plane partition.

In summary, a $(\BC^*)^4 $-fixed ideal sheaf can be described in terms of the following data:
\begin{enumerate}[label=(\roman*)]
    \item a finite-size plane partition $\lambda_{\alpha\beta}$ assigned to each edge $\alpha\beta\in E(X)$;
    \item a (possibly infinite) solid partition $\pi_\alpha$ assigned to each vertex $\alpha\in V(X)$, such that the asymptotics of $\pi_\alpha$ in the three coordinate directions is given by the plane partitions $\lambda_{\alpha\beta}$ assigned to the corresponding edges.
\end{enumerate}
Let $Z\in \Hilb^n(X, \beta)^{\TT}$ correspond to the partition data $\set{\pi_\alpha, \lambda_{\alpha\beta}}_{\alpha,\beta}$. We see 
\[\beta=\sum_{\alpha\beta\in E(X)}|\lambda_{\alpha\beta}|[L_{\alpha\beta}]\in H_2(X, \BZ), \]
where $|\lambda_{\alpha\beta}|$ denotes the size of the plane partition $\lambda$, the number of the boxes in the diagram.

For a (possibly infinite) solid partition $\pi$ such that its asymptotic plane partitions $\lambda_{i}$, $i=1,2,3,4$, are of finite size $|\lambda_i|< \infty$, we define the renormalized volume $|\pi|$ as follows.  We set 
\[
|\pi|:=\#\set{\pi\cap [0,\dots, N]^4}-(N+1)\sum_{i=1}^4|\lambda_i|, \quad N\gg 0.
\]
The renormalized volume is independent of the cut-off $N$ as long as $N$ is sufficiently large. We will say that a solid partition is \emph{point-like} if all the asymptotics $\lambda_i=0$, $i=1,\dots, 4$ and \emph{curve-like} if at least one $\lambda_i\neq 0$.  See  \cite{ORV_classical_crystals} for a similar discussion of the renormalization of  infinite \emph{plane partitions} and its interpretation  in terms of melting crystals. \\

To conclude, let $\mathbf{m}=(m_2, m_3, m_4)$, $\lambda $ a finite-size plane partition and set
\[
f_{\mathbf{m}}(\lambda)=\sum_{(i,j,k)\in \lambda }(1-m_2\cdot i-m_3\cdot j-m_4\cdot  k).
\]
By \cite[Lemma 2.4]{CK_DT-PT}, if a  $\TT$-invariant closed subscheme $Z\subset X$ corresponds to a partition data $\set{\pi_\alpha, \lambda_{\alpha\beta}}_{\alpha,\beta}$, then 
\begin{equation}\label{eqn: holom euler char with partitions}
    \chi(\oO_Z)=\sum_{\alpha\in V(X)}|\pi_\alpha|+ \sum_{\alpha\beta\in E(X)}f_{\mathbf{m}_{\alpha\beta}}(\lambda_{\alpha\beta}),
\end{equation}
where $\mathbf{m}_{\alpha\beta} $ is the multidegree of the normal bundle of the $\TT$-invariant line $L_{\alpha\beta}$.
\subsection{The vertex formalism}
Denote by $\CI$ the universal ideal sheaf of the universal subscheme $\CZ\subset X\times \Hilb^n(X, \beta)$. The obstruction theory $\RR\hom_\pi(\CI, \CI)^\vee_0[-1]\to \BL_{ \Hilb^n(X, \beta)} $ is naturally $\TT$-equivariant  by  \cite[Cor. 4.4]{Ric_equivariant_Atiyah} and endowed by the $\TT$-equivariant Serre quadratic pairing.   Over a point $Z\in \Hilb^n(X, \beta)$ the fiber of  the virtual tangent space is
\[T_Z^{\vir}= \RR\hom_\pi(\CI, \CI)_0[1]|_Z=\RR\Hom(I_Z, I_Z)_0[1],\]
where $I_Z$ is the ideal sheaf of $Z$. To a $\TT$-fixed point $Z$ corresponds a partition data $\set{\pi_\alpha, \lambda_{\alpha\beta}}$ with $\alpha\in V(X), \alpha\beta\in E(X)$; denote by 
    \begin{align}
       Z_\alpha&=H^0(Z|_{U_\alpha}, \oO_{Z|_{U_\alpha}})=\sum_{(i,j,k,l)\in \pi} t_1^it_2^jt_3^kt_4^l,\label{eqn: Z alpha}\\
    Z_{\alpha\beta}&=\sum_{(a,b,c)\in \lambda_i} t_j^a t_k^b t_l^c,\label{eqn:Z alpha beta}
\end{align}
where the line $L_{\alpha\beta}$ is defined by $\set{x_j=x_k=x_l=0}$. Recall the vertex formalism\footnote{We only write down $\mathsf{F}_{\alpha\beta}$ and $\mathsf{E}_{\alpha\beta}$ when $L_{\alpha\beta} \cong \BP^1$ is given by $\{x_2=x_3=x_4=0\}$, i.e.~the leg along the $x_1$-axis. The other cases follow by symmetry.} developed in \cite[Sec. 2.4]{CK_DT-PT}
\begin{align}
    \mathsf{V}_\alpha&=Z_\alpha+\overline{Z_\alpha}-\overline{P_{1234}}Z_\alpha\overline{Z_\alpha}+\sum_{i=1}^4\frac{\mathsf{F}_{\alpha\beta_i}(t_{i'},t_{i''},t_{i'''})}{1-t_i},\label{eqn: full vertex}\\
    \mathsf{F}_{\alpha\beta}&=-Z_{\alpha\beta}+ \frac{\overline{Z_{\alpha\beta}}}{t_2t_3t_4}+\overline{P_{234}} Z_{\alpha\beta}\overline{Z_{\alpha\beta}},\label{eqn: full Falphabeta}\\
    \mathsf{E}_{\alpha\beta}&=t_1^{-1}\frac{\mathsf{F}_{\alpha\beta}(t_2,t_3,t_4)}{1-t_1^{-1}}-\frac{\mathsf{F}_{\alpha\beta}(t_2t_1^{-m_{\alpha\beta}},t_3t_1^{-m'_{\alpha\beta}},t_4t_1^{-m''_{\alpha\beta}})}{1-t_1^{-1}}\label{eqn: full edge term}
\end{align}
where $\set{t_i, t_{i'},t_{i''},t_{i'''}}=\set{1,2,3,4}$ and for a set of indices $I$, we set $P_I=\prod_{a\in I}(1-t_a)$. This vertex formalism allows us to compute the virtual tangent space at a $\TT$-fixed point.
\begin{prop}[{\cite[Prop. 2.11]{CK_DT-PT}}]\label{prop: vertex formalism}
Let $X$ be a toric Calabi-Yau 4-fold, $\beta\in H_2(X, \BZ)$ and $Z\in  \Hilb^n(X, \beta)^\TT$. Then 
\[T_Z^{\vir}=\sum_{\alpha\in V(X)}\mathsf{V}_\alpha+ \sum_{\alpha\beta\in E(X)}\mathsf{E}_{\alpha\beta}.\]
\end{prop}
Anticipating some results from Section \ref{sec: vertex curves}, \ref{sec:edge} we prove  that the   obstruction theory induced on the $\TT$-fixed locus is trivial. 
\begin{prop}\label{prop: induced obstruction is trivial}
Let $X$ be a toric Calabi-Yau 4-fold and $\beta\in H_2(X, \BZ)$. Then the induced obstruction theory on $ \Hilb^n(X, \beta)^\TT$ is trivial. In particular, for  $Z\in  \Hilb^n(X, \beta)^\TT$, the virtual tangent space $T_Z^{\vir}$ is $\TT$-movable.
\end{prop}
\begin{proof}
By Lemma  \ref{lemma: vertex curves T-movable}, \ref{lemma: edge curves T-movable} there exist $\TT$-movable square roots $ \mathsf{v}_\alpha, \mathsf{e_{\alpha\beta}}$ of $\mathsf{V}_\alpha, \mathsf{E}_{\alpha\beta}$ for any $\alpha\in V(X), \alpha\beta\in E(X)$, which implies that also $\mathsf{V}_\alpha, \mathsf{E}_{\alpha\beta} $ and $T_Z^{\vir}$ are $\TT$-movable by Proposition \ref{prop: vertex formalism}. We have an identity in $\TT$-equivariant $K$-theory
\[ T_Z^{\vir}= \Ext^1(I_Z,I_Z)-\Ext^2(I_Z,I_Z)+\Ext^3(I_Z,I_Z)\in K_\TT^0(\pt).\]
It follows by Lemma \ref{lemma: finitely many reduced points} that $\Ext^1(I_Z,I_Z)^{\TT}=\Ext^3(I_Z,I_Z)^{\TT}=0 $, by which we conclude that $\Ext^2(I_Z,I_Z)^\TT=0$ as well.
\end{proof}
In the remainder of the paper, we will exhibit  (several) explicit square roots $\mathsf{v}_\alpha, \mathsf{e}_{\alpha\beta}$ of $\mathsf{V}_\alpha, \mathsf{E}_{\alpha\beta}$ reducing the DT invariants in \eqref{eqn: DT_inv} to
\begin{align}\label{eqn: sign with alphabeta}
\DT_n(X, \beta;\gamma)=\sum_{Z\in \Hilb^n(X, \beta)^\TT}\prod_{\alpha\in V(X)}(-1)^{\sigma(Z,\mathsf{v}_\alpha)}e^\TT(-\mathsf{v}_\alpha)   \prod_{\alpha\beta\in E(X)}(-1)^{\sigma(Z,\mathsf{e}_{\alpha\beta})} e^\TT(-\mathsf{e}_{\alpha\beta})   \cdot \gamma|_{Z},
\end{align}
and propose explicit \emph{canonical} signs $(-1)^{\sigma(Z,\mathsf{v}_\alpha)}, (-1)^{\sigma(Z,\mathsf{e}_{\alpha\beta})}$.
\subsection{The vertex term: points} \label{sec: local model points}
To each fixed point $Z\in \Hilb^n(\BC^4)^{\TT}$ corresponds a solid partition $\pi$ of size $n$. 
Denote by $Z_\pi, \mathsf{V}_\pi$ the vertex terms $Z_\alpha, \mathsf{V}_\alpha$ in (\ref{eqn: Z alpha}),(\ref{eqn: full vertex}). By $\TT$-equivariant Serre duality, we know that $\mathsf{V}_\pi$ admits a square root. We set
\[\mathsf{v}^i_\pi=Z_\pi-\overline{P_{jkl}}Z_\pi\overline{Z_\pi}, \]
which enjoys 
\[\mathsf{V}_\pi=\mathsf{v}_\pi^i+ \overline{\mathsf{v}_\pi^i}, \]
where $\set{i,j,k,l}=\set{1,2,3,4}$. For $i=4$, it recovers the square root already found in \cite{NP_colors, CKM_K_theoretic}. The next lemma was already  proven by Nekrasov-Piazzalunga  \cite{NP_colors}, whose proof we sketch for completeness and to introduce useful notation.
\begin{lemma}[{\cite[Sec. 2.4.1]{NP_colors}}]\label{lemma: vertex points T-movable}
Let $\pi$ be a point-like solid partition and $i=1,\dots 4$. Then  $\mathsf{v}^i_\pi$ is $\TT$-movable. 
\end{lemma}

\begin{proof}
Without loss of generality, suppose that $i=4$. We prove the statement by induction on the size of $\pi$. If $|\pi|=1$, then $ \mathsf{v}_\pi^4$ has no constant terms. Suppose now that the claim holds for all solid partitions $\pi$ of size $|\pi|\leq n$. Consider a solid partition $\Tilde{\pi}$ of size $|\Tilde{\pi}|=n+1$; this can be seen as a solid partition $\pi$ of size $n$ with an extra box over it, corresponding to a $\BZ^4$-lattice point $\mu=(\mu_1,\mu_2,\mu_3,\mu_4)$. We have
\[
Z_{\Tilde{\pi}}=Z_{\pi}+ t^\mu
\]
and
\begin{align*}
\mathsf{v}^4_{\Tilde{\pi}}&=\mathsf{v}^4_\pi+ t^\mu-\overline{P_{123}}(t^\mu\overline{Z_\pi}+ t^{-\mu}Z_\pi+1)\\
&=\mathsf{v}^4_\pi+t^\mu-\overline{P_{123}}(t^\mu(\overline{Z_\pi}+ t^{-\mu})+ t^{-\mu}(Z_\pi+t^\mu)-1).
\end{align*}
By induction, we know that $\mathsf{v}^4_\pi$ is $\TT$-movable. Consider the subdivision $\Tilde{\pi}=\pi^{'}\sqcup \pi^{''}$ of the boxes of $\Tilde{\pi}$, where $\pi^{'}$ corresponds to the lattice points $\nu$ such that  $\nu\leq \mu$ and $\pi^{''}$ corresponds to the lattice points $\nu$ such that  $\nu \nleq \mu$. Here, given $\nu=(\nu_1, \nu_2, \nu_3, \nu_4),$ we say that  $\nu\leq \mu$ if $\nu_i\leq \mu_i$ for all $i=1,\dots, 4$.
Denote by
\begin{align*}
    Z_{\pi^{'}}&:=\sum_{\nu\in \pi^{'}}t^\nu=\sum_{i\leq \mu_1, j\leq \mu_2, k\leq \mu_3, l\leq \mu_4} t_1^it_2^jt_3^kt_4^l,\\
    Z_{\pi^{''}}&:= \sum_{\nu\in \pi^{''}}t^\nu.
\end{align*}
By construction, $Z_{\Tilde{\pi}}= Z_{\pi^{'}}+ Z_{\pi^{''}} $. We want to prove that 
\[
\left(t^\mu+ \overline{P_{123}}- \overline{P_{123}}t^\mu \overline{Z_{\pi^{'}}}-\overline{P_{123}}t^\mu \overline{Z_{\pi^{''}}}-\overline{P_{123}}t^{-\mu}Z_{\pi^{'}}-\overline{P_{123}}t^{-\mu}Z_{\pi^{''}}\right)^{\fix}=0,
\]
where by $(\cdot)^{\fix}$ we denote the invariant factors under the $\TT$-action\footnote{ In other words, for a virtual $\TT$-representation $F\in K^0_\TT(\pt)$, $(F)^{\fix}$ is its constant term when viewing $F$ as a Laurent polynomial in the torus coordinates.}.
For a set of indices $I$, denote by $\delta_{I}$ the function which is 1 if only if all the indices are equal.  The contribution of each summand is
\begin{align*}
\left(t^\mu+ \overline{P_{123}} \right)^{\fix}&=1+\delta_{\mu_1,\mu_2,\mu_3,\mu_4},\\
\left(\overline{P_{123}}t^\mu \overline{Z_{\pi^{''}}}  \right)^{\fix}&=0,\\
\left(\overline{P_{123}}t^{-\mu}Z_{\pi^{''}}  \right)^{\fix}&=0, \\
\left(\overline{P_{123}}t^\mu \overline{Z_{\pi^{'}}}  \right)^{\fix}&=\sum_{i=0}^{\mu_4} \delta_{\mu_1,\mu_2,\mu_3,i},\\
\left(\overline{P_{123}}t^{-\mu}Z_{\pi^{'}} \right)^{\fix}&=1-\sum_{i=0}^{\mu_4-1} \delta_{\mu_1,\mu_2,\mu_3,i},
\end{align*}
by which we conclude the induction step.
\end{proof}
We propose a sign rule\footnote{After a first draft of the paper was written, Kool-Rennemo \cite{KR_magnificient} announced a proof for this sign rule.} for the sign in (\ref{eqn: sign with alphabeta}), relative to the square root $\mathsf{v}_\pi^i$.
\begin{conjecture}
Let $\pi$ be a  solid partition corresponding to a $\TT$-fixed point in $\Hilb^{n}(\BC^4) $. Then the sign relative to the square root $ \mathsf{v}_\pi^i$ is $(-1)^{\sigma_i(\pi)}$, where
\begin{align}\label{eqn: sign rule points}
    \sigma_i(\pi)=|\pi|+\#\set{(a_1,a_2,a_3,a_4)\in \pi\colon a_j=a_k=a_l<a_i}
\end{align}
and $\set{i,j,k,l}= \set{1,2,3,4}$.
\end{conjecture}
\begin{remark}
For $i=4$, the sign rule (\ref{eqn: sign rule points}) was proposed by Nekrasov-Piazzalunga in the physics literature in \cite{NP_colors}, as the result of supersymmetric localization in string theory\footnote{In \cite{NP_colors} the sign rule is slightly different due to a slightly different choice of square root.  }. It was also proposed in \cite{CKM_K_theoretic}, based on explicit low order computations. This sign rule is consistent with the previous computations of \cite{CK1, CK_higher_counting,CK_DT-PT,CKM_K_theoretic,Nek_magnificient_4, NP_colors}.
\end{remark}
\begin{remark}\label{rem: dim red points}
Let $\pi$ corresponds to a $\TT$-invariant closed subscheme $Z\subset \BC^4$ supported in the hyperplane $\set{x_i=0}\subset \BC^4$, for $i=1,\dots, 4$. Then $\sigma_i(\pi)=|\pi|$, which is consistent with the dimensional reduction studied in \cite[Sec. 2.1]{CKM_K_theoretic}. 
\end{remark}
We prove now that the sign rule (\ref{eqn: sign rule points}) is canonical, meaning that it does not really depend on choosing a preferred $x_i$-axis, as one should expect from the correct sign rule.
\begin{theorem}\label{thm: orientation points}
Let $\pi$ a point-like solid partition. For every $i,j=1,\dots, 4$ we have
\[(-1)^{\sigma_i(\pi)}e^\TT(-\mathsf{v}_\pi^i)=(-1)^{\sigma_j(\pi)}e^\TT(-\mathsf{v}_\pi^j).\]
\end{theorem}
\begin{proof}
$\mathsf{v}_\pi^i$ and $\mathsf{v}_\pi^j$ are both square roots of $\mathsf{V}_\pi$ and are $\TT$-movable by Lemma \ref{lemma: vertex points T-movable}. To prove the claim,  we need to write
\[ \mathsf{v}_\pi^i-\mathsf{v}_\pi^j=U_\pi-\overline{U_\pi},\]
for a $U_\pi\in K_\TT^0(\pt)$ and compute the parity of $\rk U_\pi^{\mov}$; in fact
\[
\frac{e^{\TT}(\mathsf{v}_\pi^i)}{e^{\TT}(\mathsf{v}_\pi^j)}=\frac{e^{\TT}(U^{\mov}_\pi)}{e^{\TT}(\overline{U^{\mov}_\pi})}=(-1)^{\rk U_\pi^{\mov} }.
\]
Without loss of generality, suppose $i=4$ and $j=3$;  we prove the claim  by induction on the size of $\pi$. If $|\pi|=1$, we have $Z_\pi=1$ and
\begin{align*}
    \mathsf{v}_\pi^4-\mathsf{v}_\pi^3&=\overline{P_{12}}(t_3^{-1}-t_4^{-1})\\
    &= \overline{P_{12}}t_3^{-1}- \overline{\overline{P_{12}}t_3^{-1}},
\end{align*}
where used that $t_1t_2t_3t_4=1$; clearly $\rk(\overline{P_{12}}t_3^{-1} )^{\mov}=0$. Suppose now the claim holds for any partition of size $|\pi|\leq n$ and consider a solid partition $\Tilde{\pi}$,  obtained by a solid partition $\pi$ of size $n$ by adding a box whose lattice coordinates are $\mu=(\mu_1,\mu_2,\mu_3,\mu_4)$. We have
  \begin{align*}
      \mathsf{v}_{\Tilde{\pi}}^4-\mathsf{v}_{\Tilde{\pi}}^3&=\mathsf{v}_\pi^4-\mathsf{v}_\pi^3+\overline{P_{12}}(t_3^{-1}-t_4^{-1})+\overline{P_{12}}(t_3^{-1}-t_4^{-1})(\overline{Z_\pi}t^\mu+ Z_\pi t^{-\mu})\\
      &=\mathsf{v}_\pi^4-\mathsf{v}_\pi^3-\overline{P_{12}}(t_3^{-1}-t_4^{-1})+\overline{P_{12}}(t_3^{-1}-t_4^{-1})(\overline{Z_{\Tilde{\pi}}}t^\mu+ Z_{\Tilde{\pi}} t^{-\mu}),
  \end{align*}
  and by the induction step
  \begin{align*}
      &e^\TT(\mathsf{v}_\pi^4-\mathsf{v}_\pi^3)=(-1)^{\sigma_4(\pi)-\sigma_3(\pi)},\\
      &e^\TT\left(\overline{P_{12}}(t_3^{-1}-t_4^{-1})\right) =1.
  \end{align*}
  The final piece to compute is of the form
  \begin{align*}
  \overline{P_{12}}(t_3^{-1}-t_4^{-1})(\overline{Z_{\Tilde{\pi}}}t^\mu+ Z_{\Tilde{\pi}} t^{-\mu})&=(W_1+\overline{W_1})(W_2-\overline{W_2})\\
  &=(W_1W_2-\overline{W_1}\overline{W_2})+ (\overline{W_1}W_2-W_1\overline{W_2}),
  \end{align*}
with 
\begin{align*}
    W_1&=t^{-\mu}Z_{\Tilde{\pi}},\\
    W_2&=t_3^{-1}\overline{P_{12}}.
\end{align*}
Since $\rk W_1W_2=\rk \overline{W_1}W_2=0$, we just need to compute the parity of $\rk(W_1W_2)^{\fix}+\rk(\overline{W_1}W_2)^{\fix}$.
As in the proof of Lemma \ref{lemma: vertex points T-movable}, consider the subdivision $\Tilde{\pi}=\pi^{'}\sqcup\pi^{''}$, with  $Z_{\Tilde{\pi}}= Z_{\pi^{'}}+ Z_{\pi^{''}} $, where 
\begin{align*}
    Z_{\pi^{'}}&=\sum_{i\leq \mu_1, j\leq \mu_2, k\leq \mu_3, l\leq \mu_4} t_1^it_2^jt_3^kt_4^l,\\
    Z_{\pi^{''}}&= \sum_{\nu\in \pi^{''}}t^\nu.
\end{align*}
We have
\begin{align*}
    \left(W_1W_2\right)^{\fix}&=\left(t_3^{-1}\overline{P_{12}}t^{-\mu} Z_{\pi^{'}}+t_3^{-1}\overline{P_{12}}t^{-\mu} Z_{\pi^{''}} \right)^{\fix}.
\end{align*}
As in the proof Lemma  \ref{lemma: vertex points T-movable} we compute
\begin{align*}
    \left( t_3^{-1}\overline{P_{12}}t^{-\mu} Z_{\pi^{''}}\right)^{\fix}&=0,\\
     \left( t_3^{-1}\overline{P_{12}}t^{-\mu} Z_{\pi^{'}}\right)^{\fix}&=\sum_{k=0}^{\mu_3}\sum_{l=0}^{\mu_4-1}\delta_{\mu_1,\mu_2,k,l}.
\end{align*}
Notice now that 
\[
\overline{W}_2=t_4^{-1}\overline{P_{12}},
\]
thus, by symmetry, we have
\[ \left(\overline{W_1}W_2 \right)^{\fix}= \sum_{k=0}^{\mu_3-1}\sum_{l=0}^{\mu_4}\delta_{\mu_1,\mu_2,k,l}.\]
We compute the parity
\begin{align*}
    \rk (W_1W_2)^{\fix}+ \rk(\overline{W_1}W_2)^{\fix}&=\sum_{k=0}^{\mu_3}\sum_{l=0}^{\mu_4-1}\delta_{\mu_1,\mu_2,k,l}+ \sum_{k=0}^{\mu_3-1}\sum_{l=0}^{\mu_4}\delta_{\mu_1,\mu_2,k,l}\\
    &= \sum_{l=0}^{\mu_4-1}\delta_{\mu_1,\mu_2,\mu_3,l}+ \sum_{k=0}^{\mu_3-1}\delta_{\mu_1,\mu_2,k,\mu_4}\mod 2.
\end{align*}
 Notice that 
\[
\sum_{l=0}^{\mu_4-1}\delta_{\mu_1,\mu_2,\mu_3,l}=\begin{cases} 1& \mu_4>\mu_1=\mu_2=\mu_3\\
0& \mbox{else}
\end{cases}
\]
therefore $\sigma_4(\Tilde{\pi})=\sigma_4(\pi)+ \sum_{l=0}^{\mu_4-1}\delta_{\mu_1,\mu_2,\mu_3,l}$ and  $\sigma_3(\Tilde{\pi})=\sigma_3(\pi)+ \sum_{k=0}^{\mu_3-1}\delta_{\mu_1,\mu_2,k,\mu_4}$, by which we conclude the proof.
\end{proof}
\subsection{The vertex term: curves}\label{sec: vertex curves}
Set $\underline{\lambda}=(\lambda_1,\lambda_2,\lambda_3,\lambda_4)$, where $\lambda_i$ are finite plane partitions. Denote by $P_{\underline{\lambda}}$ the collection of (possibly infinite) solid partitions, whose asymptotic profile is given by $\underline{\lambda}$. To any solid partition $\pi\in P_{\underline{\lambda}}$ corresponds a $\TT$-invariant closed subscheme $Z\subset \BC^4$;  denote  by $Z_{\pi}, Z_{\lambda_i}, \mathsf{V}_\pi$ the vertex terms $Z_\alpha, Z_{\alpha\beta_i}, \mathsf{V}_\alpha$ in (\ref{eqn: Z alpha}), (\ref{eqn:Z alpha beta}), (\ref{eqn: full vertex}). $\TT$-equivariant Serre duality implies that $\mathsf{V}_\pi $ admits a square root; set
\begin{align*}
    \mathsf{v}^i_\pi&=  Z_\pi  - \overline{P_{jkl}} Z_\pi \overline{Z_\pi}  + \sum_{j\neq i, j=1}^4 \frac{ \mathsf{f}^i_{\lambda_j}}{1-t_j} + \frac{1}{(1-t_i)} \left( - Z_{\lambda_i} + \overline{P_{jkl}} \left( \overline{Z_\pi} Z_{\lambda_i} - Z_\pi \overline{Z_{\lambda_i}} \right) + \frac{\overline{P_{jkl}}}{1-t_i} Z_{\lambda_i} \overline{Z_{\lambda_i}} \right), \\
    \mathsf{f}^i_{\lambda_j}&=-Z_{\lambda_j}   + \overline{P_{kl}} Z_{\lambda_j} \overline{Z_{\lambda_j}},
\end{align*}
which enjoy
\begin{align*}
    \mathsf{V}_\pi&=\mathsf{v}^i_\pi+\overline{\mathsf{v}^i_\pi},
\end{align*}
where $\set{i,j,k,l}=\set{1,2,3,4}$. Setting $i=4$, we recover the explicit square root choice studied in \cite{CKM_K_theoretic}. We redistribute now the vertex terms in a different way.  Write, for $N\gg 0$
\begin{align*}
Z_\pi&=Z_{\pi^{\nor}}+ \sum_{i=1}^4  \frac{Z_{\lambda_i}}{1-t_i}t_i^{N+1},\\
Z_{\pi_i}&= Z_{\lambda_i}\sum_{n=0}^N t_i^n,\\
\frac{Z_{\lambda_i}}{1-t_i}&=Z_{\pi_i}+ \frac{Z_{\lambda_i}}{1-t_i}t_i^{N+1}.
\end{align*}
Here the (point-like) solid partition $\pi^{\nor}$  is the cut-off for $N\gg 0$ of the (possibly curve-like) solid partition $\pi$. If $\pi$ is a point-like solid partition, then simply $\pi^{\nor}=\pi $, while $\pi_i$ is simply the cut-off for $N\gg 0$ of the curve-like solid partition corresponding to the infinite leg along the $x_i $-axis containing $\pi$. Using the above expressions, we can express the vertex terms as
\begin{align}
    \mathsf{v}_\pi^i=\mathsf{v}^i_{\pi^{\nor}}- \sum_{a=1}^4\mathsf{v}^i_{\pi_a}+A^i_\pi+B^i_\pi+C^i_\pi-\overline{C^{i}_\pi},
\end{align}
where 
\begin{align*}
    A^i_\pi&=-\overline{P_{jkl}}\sum_{\substack{a\neq b\\ a,b\neq i}}\frac{Z_{\lambda_a}}{1-t_a}\frac{\overline{Z_{\lambda_b}}}{1-t^{-1}_b}(t_at_b^{-1})^{N+1}-\overline{P_{1234}}\sum_{a\neq i}\frac{Z_{\lambda_a}}{1-t_a}\frac{\overline{Z_{\lambda_i}}}{1-t^{-1}_i}(t_at_i^{-1})^{N+1},\\
    B^i_\pi&=-\overline{P_{jkl}}\sum_{a\neq i}\left( \frac{Z_{\lambda_a}}{1-t_a}t_a^{N+1}(\overline{Z_{\pi^{\nor}}}-\overline{Z_{\pi_a}})+ \frac{\overline{Z_{\lambda_a}}}{1-t^{-1}_a}t_a^{-(N+1)}(Z_{\pi^{\nor}}-Z_{\pi_a})\right)\\
    &\qquad \qquad \qquad \qquad \qquad \qquad \qquad \qquad \qquad \qquad- \overline{P_{1234}}\frac{\overline{Z_{\lambda_i}}}{1-t^{-1}_i}t_i^{-(N+1)}(Z_{\pi^{\nor}}-Z_{\pi_i}),\\
    C^i_\pi&= \overline{P_{123}}\left(Z_{\pi_i}(\overline{Z_{\pi^{\nor}}}-\overline{Z_{\pi_i}})+\sum_{a\neq i}  \frac{\overline{Z_{\lambda_a}}}{1-t^{-1}_a}t_a^{-(N+1)}Z_{\pi_i} \right),
\end{align*}
where  $\set{i,j,k,l}=\set{1,2,3,4}$. Motivated by the above expression, we define a new square root of $\mathsf{V}_\pi$
\[
{\mathsf{v}_\pi^{i}}'=\mathsf{v}_\pi^i- C^i_\pi+\overline{C^i_\pi}.
\]
\begin{remark}
It is an  equivalent problem to find the sign rule for $ {\mathsf{v}_\pi^{i}}'$ or for $\mathsf{v}_\pi^{i}$. In fact, such two sign rules will differ just by $(-1)^{\rk(C_\pi^i )^{\mov}}$, which is completely determined by the solid partition $\pi$. Moreover, if the $\TT$-invariant closed subscheme $Z\subset \BC^4$ corresponding to a solid partition $\pi $ is not supported in the $\TT$-invariant line $\set{x_j=x_k=x_l=0}$, we have that ${\mathsf{v}_\pi^{i}}'=\mathsf{v}_\pi^{i} $.
\end{remark}
\begin{lemma}
\label{lemma: vertex curves T-movable}
Let $\pi$ be a curve-like solid partition and $i=1,\dots, 4$. Then  $\mathsf{v}^i_\pi,{\mathsf{v}_\pi^{i}}'$ are $\TT$-movable. 
\end{lemma}
\begin{proof}
By Lemma \ref{lemma: vertex points T-movable} $\mathsf{v}^i_{\pi^{\red}}$, $\mathsf{v}^i_{\pi_a}$ are $\TT$-movable, for $a=1,\dots, 4$. For $N\gg 0$, we clearly have $(A^i_\pi)^{\fix}=(B^i_\pi)^{\fix}=0$.
\end{proof}
We propose a sign rule for the sign in (\ref{eqn: sign with alphabeta}), relative to the square root ${\mathsf{v}_\pi^{i}}'$.
\begin{conjecture}\label{conj: sign vertex curves}
Let $\pi$ be a  curve-like solid partition. Then the sign relative to the square root $ {\mathsf{v}_\pi^{i}}'$ is $(-1)^{\sigma_i(\pi)}$, where
\begin{multline}\label{eqn: sign rule curves}
    \sigma_i(\pi)=|\pi|+\#\set{(a_1,a_2,a_3,a_4)\in \pi\colon a_j=a_k=a_l<a_i}\\-\sum_{\mathrm{leg}}\#\set{(a_1,a_2,a_3,a_4)\in \mathrm{leg}\colon a_j=a_k=a_l<a_i}
\end{multline}
and $\set{i,j,k,l}= \set{1,2,3,4}$, where $\mathrm{leg} $ denote the curve-like solid partitions obtained by translating the plane partitions $\lambda_i$ along the $x_i$-axis.
\end{conjecture}
\begin{remark}\label{rmk: dim red curves vertex}
For $i=4$ the sign rule (\ref{eqn: sign rule curves}) was proposed in \cite{CKM_K_theoretic} for at most two non-empty legs, based on explicit low order computations. This sign rule is consistent with the computations of \cite{CK_DT-PT,CKM_K_theoretic}.
\end{remark}
\begin{remark}\label{rem: dim red curves}
Let $\pi$ corresponds to a $\TT$-invariant closed subscheme $Z\subset \BC^4$ supported in the hyperplane $\set{x_i=0}\subset \BC^4$, for $i=1,\dots, 4$. Then $\sigma_i(\pi)=|\pi|$, which is consistent with the dimensional reduction studied in \cite[Sec. 2.1]{CKM_K_theoretic}. 
\end{remark}
We prove now that the sign rule (\ref{eqn: sign rule curves}) is canonical, meaning that it does not really depend on choosing a preferred $x_i$-axis.
\begin{theorem}\label{thm: orientation curves}
Let $\pi$ a curve-like solid partition. For every $i,j=1,\dots, 4$ we have
\[(-1)^{\sigma_i(\pi)}e^\TT(-{\mathsf{v}_\pi^{i}}')=(-1)^{\sigma_j(\pi)}e^\TT(-{\mathsf{v}_\pi^{j}}').\]
\end{theorem}
\begin{proof}
${\mathsf{v}_\pi^i}'$ and ${\mathsf{v}_\pi^j}'$ are both square roots of $\mathsf{V}_\pi$ and are $\TT$-movable by Lemma \ref{lemma: vertex curves T-movable}.
The difference of the two sections is
\begin{align*}
   {\mathsf{v}_\pi^i}'-{\mathsf{v}_\pi^j}'&=(\mathsf{v}^i_{\pi^{\nor}}-\mathsf{v}^j_{\pi^{\nor}} )- \sum_{a=1}^4 (\mathsf{v}^i_{\pi_a}-\mathsf{v}^j_{\pi_a} )+A^i_\pi-A^j_\pi+B^i_\pi-B^j_\pi.
\end{align*}
By Theorem \ref{thm: orientation points}, we know that 
\begin{align*}
   e^\TT(\mathsf{v}^i_{\pi^{\nor}}-\mathsf{v}^j_{\pi^{\nor}})&=(-1)^{\sigma_i(\pi^{\nor})-\sigma_j(\pi^{\nor})},\\
    e^\TT(\mathsf{v}^i_{\pi_a}-\mathsf{v}^j_{\pi_a})&=(-1)^{\sigma_i(\pi_a)-\sigma_j(\pi_a)},
\end{align*}
which satisfies 
\begin{align*}
   \sigma_i(\pi)-\sigma_j(\pi)=\sigma_i(\pi^{\nor})-\sigma_j(\pi^{\nor})-\sum_{a=1}^4(\sigma_i(\pi_a)-\sigma_j(\pi_a) ).
\end{align*}
To conclude, with a simple computation it is possible to show that
\begin{align*}
    A_\pi^i-A_\pi^j&=U_{A,i,j,\pi}-\overline{U_{A,i,j,\pi}},\\
    B_\pi^i-B_\pi^j&=U_{B,i,j,\pi}-\overline{U_{B,i,j,\pi}},  
\end{align*}
where, for $N\gg 0$, we have $\rk(U_{A,i,j,\pi})^{\mov}=\rk(U_{A,i,j,\pi})^{\fix}=0$ and $\rk(U_{B,i,j,\pi})^{\mov}=\rk(U_{B,i,j,\pi})^{\fix}=0$.
\end{proof}
\subsection{The edge term}\label{sec:edge}
In this section we study square roots and propose a sign rule for the edge term \ref{eqn: full edge term}. For simplicity, let's assume that the edge $\alpha\beta\in E(X)$ corresponds to the $\BP^1$ given, in the local coordinates of $U_\alpha$, by $\set{x_2=x_3=x_4=0}$, with normal bundle
\[  
N_{\BP^1/X}\cong \oO(m_2)\oplus\oO(m_3)\oplus\oO(m_4),
\]
satisfying $m_2+m_3+m_4=-2$; we set $\mathbf{m}=(m_2,m_3,m_4)$. We also fix  a finite-size plane partition $\lambda$, corresponding to the profile of the non-reduced $\TT$-fixed leg. The discussion for the legs along the other directions will be completely analogous. Denote here by $Z_\lambda, \mathsf{E}_{\lambda}$ the edge terms $Z_{\alpha\beta}, \mathsf{E}_{\alpha\beta}$ in (\ref{eqn:Z alpha beta}), (\ref{eqn: full edge term}), where 
\begin{align*}
    Z_{\lambda}&=\sum_{(j,k,l)\in \lambda} t_2^j t_3^k t_4^l.
\end{align*}
Denote by $\widetilde{(\cdot)}:K_\TT^0(\pt)\to K_\TT^0(\pt)$ the map sending 
\[V\mapsto V(t_1^{-1}, t_2t_1^{-m_2},t_3t_1^{-m_3},t_4t_1^{-m_4}).\]
The edge term admits a square root; set
\begin{align*}
    \mathsf{e}^j_\lambda&= t_1^{-1}\frac{   \mathsf{f}^j_{\lambda}}{1-t_1^{-1}}- \frac{ \widetilde{\mathsf{f}^j_{\lambda}}}{1-t_1^{-1}},\\
    \mathsf{f}^j_{\lambda}&=-Z_{\lambda}   + \overline{P_{kl}} Z_{\lambda} \overline{Z_{\lambda}},
\end{align*}
which enjoys
\begin{align*}
    \mathsf{E}_\lambda&=\mathsf{e}^j_\lambda+\overline{\mathsf{e}^j_\lambda},
\end{align*}
where $\set{j,k,l}=\set{2,3,4}$. Setting $j=4$, we recover the explicit square root choice studied in \cite{CKM_K_theoretic}.
\begin{lemma}
\label{lemma: edge curves T-movable}
Let $\lambda$ be a finite plane partition and $j=2,3, 4$. Then  $\mathsf{e}^j_\lambda$ is $\TT$-movable. 
\end{lemma}
\begin{proof}
Without loss of generality, suppose that $j=4$. Write $\mathsf{f}^4_{\lambda}=\sum_\nu t^\nu $, where the sum is over $\nu=(\nu_2, \nu_3, \nu_4)$ and we set  $t^\nu=t_2^{\nu_2}t_3^{\nu_3}t_4^{\nu_4} $. We have
\begin{align}\label{eqn: split with no pole}
     \frac{1}{1-t_1^{-1}}(t_1^{-1}t^{\nu}-t^{\nu}t_1^{-\mathbf{m}\nu})=\begin{cases} -t^{\nu}\sum_{i=0}^{-\mathbf{m}\nu}t_1^{i}& \mathbf{m}\nu \leq 0,\\
    0 & \mathbf{m}\nu=1,\\
    t^{\nu}t_1^{-1}\sum_{i=0}^{\mathbf{m}\nu-2}t_1^{-i}& \mathbf{m} \nu\geq 2\\
    \end{cases}
\end{align}
   where $\mathbf{m} \nu$ denotes the standard scalar product in $\BZ^3$. Therefore the contribution to the $\TT$-fixed part of each $t^{\nu}$ is
 \begin{align}\label{eqn: edge mov piece dim}
        \rk\left( \frac{1}{1-t_1^{-1}}(t_1^{-1}t^{\nu}-t^{\nu}t_1^{-\mathbf{m}\nu}) \right)^{\fix}&=\begin{cases}
     -\sum_{i=0}^{-\mathbf{m}\nu}\delta_{i,\nu_2, \nu_3,\nu_4}& \mathbf{m}\nu \leq 0,\\
     0 &  \mathbf{m}\nu=1,\\
     \sum_{i=1-\mathbf{m}\nu}^{-1}\delta_{i,\nu_2, \nu_3,\nu_4}& \mathbf{m}\nu \geq 2\\
    \end{cases}\\
    &=\begin{cases}
    -1& \nu_2=\nu_3=\nu_4\geq 0,\\
    1& \nu_2=\nu_3=\nu_4\leq  -1,\\
     0 & \mbox{else}.\\
    \end{cases}
   \end{align}
   Denote by $W_l$ the sub-representation of $\mathsf{f}_{\lambda}^4$ corresponding to the irreducible $\TT$-representation $(t_2t_3t_4)^l$, for $l\in \BZ$. Equation (\ref{eqn: edge mov piece dim}) translates into
   \[
    \rk\sum_{\nu}\left( \frac{1}{1-t_1^{-1}}(t_1^{-1}t^{\nu}-t^{\nu}t_1^{-\mathbf{m}\nu}) \right)^{\fix}=\sum_{l\geq 0}\left(\rk W_{-l-1}-\rk W_{l}\right).
   \]
   Notice that $\mathsf{f}_{\lambda}^4-\overline{\mathsf{f}_{\lambda}^4}(t_2t_3t_4)^{-1} $ is the 3-fold vertex of \cite[Eqn. (12)]{MNOP_1} in the variables $t_2,t_3,t_4$, which  is $\TT_0$-movable for $\TT_0=\set{t_2t_3t_4=1}\subset (\BC^*)^3$ (cf. \cite[pag. 1279]{MNOP_1}). This implies that for any $l\in \BZ$
   \[
   \rk W_{l}=\rk W_{-l-1},
   \]
   by which we conclude the proof.
\end{proof}
We propose a sign rule for the sign in (\ref{eqn: sign with alphabeta}), relative to the square root ${\mathsf{e}_\lambda^{i}}$.
\begin{conjecture}\label{conj: sign edge}
Let $\lambda$ be a finite plane partition. Then the sign relative to the square root $ \mathsf{e}_\lambda^i$ is $(-1)^{\sigma_i(\lambda)}$, where
\begin{equation}
    \label{eqn: sign rule edge}
    \sigma_i(\lambda)= f_{\mathbf{m}}(\lambda)+|\lambda|m_i+ \#\set{(a_2,a_3,a_4)\in \lambda\colon a_j=a_k<a_i},
\end{equation}
  and  $\set{i,j,k}= \set{2,3,4}$.
\end{conjecture}
We prove now that the sign rule (\ref{eqn: sign rule edge}) is canonical, meaning that it does not really depend on choosing a preferred $x_i$-axis.
\begin{theorem}\label{thm: orientation edge}
Let $\lambda$ be  a finite plane partition. For every $i,j=2,3, 4$ we have
\[(-1)^{\sigma_i(\lambda)}e^\TT(-\mathsf{e}_\lambda^i)=(-1)^{\sigma_j(\lambda)}e^\TT(-\mathsf{e}_\lambda^j).\]
\end{theorem}
\begin{proof}
Without loss of generality assume $i=4, j=3$. Say, for $k\in \BZ$,
\[
A(k)=\begin{cases} -\sum_{i=0}^{-k}t_1^{i}&k \leq 0,\\
    0 & k=1,\\
    t_1^{-1}\sum_{i=0}^{k-2}t_1^{-i}& k\geq 2\\
    \end{cases}
\]
and, for a  $\TT$-representation $V$,
\[
B(V)=\sum_{\nu\in V}t^{\nu}A(\mathbf{m}\nu)\in K^0_\TT(\pt),
\]
where the sum is over the weight spaces of $V$. We extend the definition of $B(V)$ by linearity to $K^0_\TT(\pt)$.  By (\ref{eqn: split with no pole}), we have
\begin{align*}
   \mathsf{e}_\lambda^4-\mathsf{e}_\lambda^3=B(\mathsf{f}_\lambda^4-\mathsf{f}_\lambda^3),
\end{align*}
Notice the decomposition 
\begin{align*}
  \mathsf{f}_\lambda^4-\mathsf{f}_\lambda^3&=W_\lambda+\overline{W_\lambda}(t_2t_3t_4)^{-1},\\
  W_\lambda&=Z_\lambda \overline{Z_\lambda}(t_4^{-1}-t_3^{-1}).
\end{align*}
Then 
\begin{align*}
   \mathsf{e}_\lambda^4-\mathsf{e}_\lambda^3&= B(W_\lambda)+B(\overline{W_\lambda}(t_2t_3t_4)^{-1})\\
   &=B(W_\lambda)-\overline{B(W_\lambda)},
\end{align*}
by which   we conclude that 
\[e^\TT(  \mathsf{e}_\lambda^4-\mathsf{e}_\lambda^3)=(-1)^{\rk{ B(W_\lambda)}^{\mov}}.\]
We compute the parity of $\rk{ B(W_\lambda)}^{\mov}$ by induction on the size of $\lambda$. If $|\lambda|=1$, we clearly have 
\[ \rk{B(W_\lambda)}^{\mov}=m_4+m_3  \mod 2.\]
Suppose now that the claim holds for all plane partition of size $|\lambda|\leq n$ and consider a plane partition $\Tilde{\lambda}$ of size $|\Tilde{\lambda}|=n+1$; this can be seen as a plane partition $\lambda$ of size $n$ with an extra box over it, corresponding to a $\BZ^3$-lattice point $\mu=(\mu_2,\mu_3,\mu_4)$.  We have
\begin{align*}
    B(W_{\Tilde{\lambda}} )&=B(W_{\lambda})+B(Y_4)-B(Y_3)+ B(t_4^{-1}-t_3^{-1}),\\
    Y_i&=t_i^{-1}(Z_\lambda t^{-\mu}+ \overline{Z_{\lambda}}t^\mu) \quad i=3,4,
\end{align*}
and  by the inductive step 
\begin{align*}
    \rk B(W_{\lambda})^{\mov}&=\sigma_4(\lambda)-\sigma_3(\lambda)\mod 2,\\
    \rk B(t_4^{-1}-t_3^{-1})^{\mov}&=m_4-m_3 \mod 2.
\end{align*}
Clearly, $\rk B(Y_4)^{\mov}=\rk B(Y_4)^{\fix} \mod 2$. In fact, 
\begin{align*}
    \rk B(Y_4)&=\sum_{\nu\in Z_{\lambda}}\left(\mathbf{m}(\mu-\nu+(0,0,-1))+\mathbf{m}(\nu-\mu+(0,0,-1)) \right)\\
    &=-2m_4|\lambda |.
\end{align*}
A simple analysis of $B(Y_4)^{\fix}$ as in (\ref{eqn: edge mov piece dim}) yields
\begin{multline}
  \rk (B(Y_4))^{\fix}=\#\set{(\nu\in\lambda \colon \mu_2-\nu_2=\mu_3-\nu_3=\mu_4-\nu_4+1}\\-\#\set{\nu\in\lambda \colon \mu_2-\nu_2=\mu_3- \nu_3=\mu_4- \nu_4-1},   
\end{multline}
 where $\nu=(\nu_2, \nu_3,\nu_4)$; in particular, it has to satisfy $\nu\leq \mu$. Therefore we can write it as
 \begin{align*}
     \rk (B(Y_4))^{\fix}&= \sum_{i=0}^{\mu_2}\sum_{j=0}^{\mu_3}\sum_{k=0}^{\mu_4}\left(\delta_{\mu_2-i, \mu_3-j, \mu_4-k+1}-\delta_{\mu_2-i, \mu_3-j, \mu_4-k-1}\right) \\
     &= \sum_{i=0}^{\mu_2}\sum_{j=0}^{\mu_3}\left(\sum_{k=-1}^{{\mu_4}-1}\delta_{{\mu_2}-i, {\mu_3}-j, {\mu_4}-k}-\sum_{k=1}^{{\mu_4}+1}\delta_{{\mu_2}-i, {\mu_3}-j, {\mu_4}-k}\right)
 \end{align*}
By symmetry we may compute the difference
\begin{multline*}
    \rk (B(Y_4)-B(Y_3) )^{\fix}=\sum_{i=0}^{\mu_2}\left(\sum_{k=-1}^{\mu_4-1}\delta_{\mu_2-i,0,\mu_4-k}+ \sum_{j=0}^{{\mu_3}-1}\delta_{\mu_2-i,\mu_3-j,\mu_4+1}- \sum_{j=-1}^{\mu_3-1}\delta_{\mu_2-i,\mu_3-j,0}\right. \\
 \left. -\sum_{k=0}^{\mu_4-1}\delta_{\mu_2-i,\mu_3+1,\mu_4-k}   -\sum_{k=1}^{\mu_4+1}\delta_{\mu_2-i,\mu_3,\mu_4-k}- \sum_{j=1}^{\mu_3}\delta_{\mu_2-i,\mu_3-j,-1}+\sum_{j=1}^{\mu_3+1}\delta_{\mu_2-i,\mu_3-j,\mu_4}+\sum_{k=1}^{\mu_4}\delta_{\mu_2-i,-1,\mu_4-k} \right)
\end{multline*}
Further analyzing which of these sums actually contribute to the rank, we finally get that
\begin{align*}
      \rk (B(Y_4)-B(Y_3) )^{\fix}= \begin{cases} 1 & \mu_2=\mu_3< \mu_4,\\ 1 & \mu_2=\mu_4< \mu_3,\\ 0 & \mbox{else}.
     \end{cases} \mod 2
\end{align*}
Therefore we conclude that
\[
\rk B(W_{\Tilde{\lambda}})^{\mov}=\sigma_4({\Tilde{\lambda}})-\sigma_3({\Tilde{\lambda}})\mod 2,
\]
which finishes the inductive step.\end{proof}
\begin{remark}\label{rem: global sign patching}
Let $X=K_Y$ be the canonical bundle of a smooth projective toric  3-fold $Y$ and consider $\Hilb^n(X,\beta)$, where $\beta\in H_2(X,\BZ)$ is a class pulled-back from $Y$. Consider a $\TT$-fixed point $Z\in \Hilb^n(X,\beta)^\TT $ (corresponding to a partition data $\set{\pi_\alpha, \lambda_{\alpha\beta}}_{\alpha,\beta}$)  scheme-theoretically supported on the zero section of $X\to Y$. Locally on the toric charts,  label the fiber direction by $x_4$ and denote by $m_{\alpha\beta}''$ the degree of the normal bundle of $L_{\alpha\beta}$ in the $x_4$-direction. Consider the square root of $T_Z^{\vir}$ given by 
\[\mathsf{v}_Z=\sum_{\alpha\in V(X)}\mathsf{v}_\alpha^4+ \sum_{\alpha\beta\in E(X)} \mathsf{e}_{\alpha\beta}^4.\]
By Remark \ref{rmk: dim red curves vertex}, the sign rules proposed for vertex and edge terms imply that the correct sign would be
\begin{align*}
    (-1)^{\sigma(Z,\mathsf{v}_Z) }&=\prod_{\alpha\in V(X)}(-1)^{\sigma_4(\pi_\alpha)}\cdot\prod_{\alpha\beta\in E(X)}(-1)^{\sigma_4(\lambda_{\alpha\beta})}\\
    &=\prod_{\alpha\in V(X)}(-1)^{|\pi_\alpha|}\cdot \prod_{\alpha\beta\in E(X)}(-1)^{|\lambda_{\alpha\beta}|m_{\alpha\beta}''+ f_{\mathbf{m}_{\alpha\beta}}(\lambda_{\alpha\beta})}\\
    &= (-1)^{n+ c_1(Y)\cdot \beta},
\end{align*}
where the last equality follows from (\ref{eqn: holom euler char with partitions}) and
\begin{align*}
    \sum_{\alpha\beta\in E(X)}|\lambda_{\alpha\beta}|m_{\alpha\beta}''&=-     \sum_{\alpha\beta\in E(X)}|\lambda_{\alpha\beta}|\deg {N_{Y/X}}_{|_{L_{\alpha\beta}}}\\
    &= -     \sum_{\alpha\beta\in E(X)}|\lambda_{\alpha\beta}|\deg {K_Y}_{|_{L_{\alpha\beta}}}\\
    &= \sum_{\alpha\beta\in E(X)}|\lambda_{\alpha\beta}| c_1(T_Y)\cdot [L_{\alpha\beta}]\\
    &= c_1(Y)\cdot \beta.
\end{align*}
 The same sign was proposed in a similar setting for  stable pair invariants \cite[ Prop. 4.2, Rmk. A.2]{CKM_Stable_Pairs}, where such local geometries are studied, motivated by a choice of preferred orientation as in \cite{Cao_GV_inv_II}. 
\end{remark}

\section{Refinements}
\subsection{$K$-theoretic invariants}
Let $X$ be a smooth  Calabi-Yau 4-fold and $M$ a moduli space of compactly supported sheaves on $X$. Oh-Thomas defined \cite[Def. 5.9]{OT_1} a \emph{(twisted) virtual structure sheaf}
\[\widehat{\oO}^{\vir}_M\in K_0(M),\]
which depends on a chosen orientation and whose properties mimic the virtual structure sheaf with the Nekrasov-Okounkov twist in the classical 3-fold theory \cite{NO_membranes_and_sheaves}. While in the 3-fold theory the twist is introduced to make  DT invariants more symmetric, here it is actually necessary to be defined. We define the $K$-theoretic version of the invariants  (\ref{eqn: DT_inv}) by means of Oh-Thomas $K$-theoretic virtual localization theorem (cf. \cite[Thm. 7.3]{OT_1}); such invariants are studied in \cite{Nek_magnificient_4, NP_colors, CKM_K_theoretic, Bojko_wall-crossing, Bojko_thesis, BFTZ_ADHM8D}.
\begin{definition}
   Let $V\in K_0^\TT(\Hilb^n(X,\beta))$. The $K$-theoretic $\TT$-equivariant Donaldson-Thomas invariants of $X$ are 
   \begin{equation*}\label{eqn: DT_inv K-th}
        \DT^K_n(X,\beta;V)=\sum_{Z\in \Hilb^n(X,\beta)^\TT}\sqrt{\mathfrak{e}^\TT}(-T^{\vir}_Z)\cdot V|_Z\in \frac{\BQ(t_1, t_2, t_3, t_4)}{(t_1t_2t_3t_4-1)}.
   \end{equation*}
\end{definition}
Here $\mathfrak{e}^\TT$ is the ($\TT$-equivariant) $K$-theoretic Euler class, which is defined as follows. Let $X$ be a scheme and $V$ a $\TT$-equivariant locally free sheaf on $X$. We define
\[\mathfrak{e}^{\TT}(V):= \Lambda^\bullet V^\vee=\sum_{i\geq 0} (-1)^i \Lambda^i V^\vee\in K^0_{\TT}(X),\]
and extend it  by linearity to any class $V\in K^0_\TT(X)$. Finally, $ \sqrt{\mathfrak{e}^\TT}(\cdot)$ is the ($\TT$-equivariant) $K$-theoretic square-root Euler  class; the complete description and construction of this class is in \cite[Sec. 5.1]{OT_1}.
Let $ V$ be a $\TT$-representation with a   square root  $T$ in $K^0_\TT(\pt)$. Its $K$-theoretic square root Euler class satisfies
\begin{align*}
    \sqrt{\mathfrak{e}^\TT}(V)&=\pm \mathfrak{e}^\TT(T)\otimes ({\det}T)^{\frac{1}{2}} \in K^0_\TT\left(\pt, \BZ\left[\frac{1}{2}\right]\right).
\end{align*}
For an irreducible $\TT$-representation\footnote{To be precise, we should replace the torus $\TT$ with its double cover where the character $t^{\frac{\mu}{2}}$ is well-defined (cf. \cite[Sec. 7.1.2]{NO_membranes_and_sheaves}).} $t^\mu$,  define
\[
[t^\mu]=t^{\frac{\mu}{2}}-t^{-\frac{\mu}{2}}\in  K^0_\TT(\pt)
\]
and extend it by linearity to any  $V\in K_\TT^0(\pt)$. It is shown in \cite[Sec. 6.1]{FMR_higher_rank} that 
\[
\mathfrak{e}^\TT(V)\otimes ({\det}V)^{\frac{1}{2}}=[V],
\]
for any virtual $\TT$-representation $V\in K_\TT^0(\pt)$. Therefore, given  square roots $\mathsf{v}_\alpha, \mathsf{e}_{\alpha\beta}$ of $\mathsf{V}_\alpha, \mathsf{E}_{\alpha\beta}$, we have
\begin{align*}
\DT^K_n(X, \beta;V)=\sum_{Z\in \Hilb^n(X, \beta)^\TT}\prod_{\alpha\in V(X)}(-1)^{\sigma(Z,\mathsf{v}_\alpha)}[-\mathsf{v}_\alpha]   \prod_{\alpha\beta\in E(X)}(-1)^{\sigma(Z,\mathsf{e}_{\alpha\beta})} [-\mathsf{e}_{\alpha\beta}]  \cdot V|_{Z}.
\end{align*}
The operator $[\cdot]$ satisfies $[t^{-\mu}]=-[t^\mu]$, the same multiplicative property of $e^\TT(\cdot)$; therefore, the results of Theorem \ref{thm: orientation points}, \ref{thm: orientation curves} and \ref{thm: orientation edge} hold the same replacing $e^\TT(\cdot)$ by $[\cdot]$.
\subsection{Elliptic invariants}
An elliptic refinement of DT invariants was proposed  in \cite[Sec. 8.2]{FMR_higher_rank} and studied in \cite{Bojko_wall-crossing}. Set
 \begin{align*}
    \theta(p;y)&=-ip^{1/8}(y^{1/2}-y^{-1/2})\prod_{n=1}^\infty(1-p^n)(1-yp^n)(1-y^{-1}p^n), \\
    \eta(p)&=p^{\frac{1}{24}}\prod_{n\ge 1}(1-p^n).
 \end{align*}
 Set $p=e^{2\pi i \tau}$, with $\tau\in \mathbb{H}=\set{\tau\in \BC | \mathrm{Im}(\tau)>0}$. Denoting $\theta(\tau|z):=\theta(e^{2\pi i\tau};e^{2\pi iz})$,  $\theta$ enjoys the  modular behaviour 
\[
\theta(\tau|z+a+b\tau)=(-1)^{a+b}e^{-2\pi ibz}e^{-i\pi b^2\tau}\theta(\tau|z),\quad a,b\in\BZ.
\]
See \cite[Sec. 8.1]{FMR_higher_rank} and \cite[Sec. 6]{FG_riemann_roch} for related discussion on the modularity of these functions.
    For an irreducible $\TT$-representation $t^\mu$,  define
     \[\theta[t^{\mu}]=  (i\cdot \eta(p))^{-1}\theta(p;t^\mu)\in  K^0_\TT(\pt)\llbracket p  \rrbracket[p^{\pm \frac{1}{12}}]\] 
   and extend it by linearity to any  $V\in K_\TT^0(\pt)$.    Let $V\in K_0^\TT(\Hilb^n(X,\beta))$. The \emph{elliptic} $\TT$-equivariant Donaldson-Thomas invariants of $X$ are 
    \begin{equation*}
      \DT^{ell}_n(X, \beta;V)=\sum_{Z\in \Hilb^n(X, \beta)^\TT}\prod_{\alpha\in V(X)}(-1)^{\sigma(Z,\mathsf{v}_\alpha)}\theta[-\mathsf{v}_\alpha]   \prod_{\alpha\beta\in E(X)}(-1)^{\sigma(Z,\mathsf{e}_{\alpha\beta})} \theta[-\mathsf{e}_{\alpha\beta}]  \cdot V|_{Z}.
   \end{equation*}
   Again, as $\theta[t^{-\mu}]=-\theta[t^\mu]$,  the results of Theorem \ref{thm: orientation points}, \ref{thm: orientation curves} and \ref{thm: orientation edge} hold the same replacing $e^\TT(\cdot)$ by $\theta[\cdot]$.
\bibliographystyle{amsplain-nodash}
\bibliography{The_Bible}
\end{document}

%% file: macros.tex
\usepackage[colorinlistoftodos]{todonotes}

\definecolor{antiquewhite}{rgb}{0.98, 0.92, 0.84}
\definecolor{buff}{rgb}{0.94, 0.86, 0.51}
\definecolor{palecopper}{rgb}{0.85, 0.54, 0.4}
\definecolor{fluorescentyellow}{rgb}{0.8, 1.0, 0.0}

\definecolor{britishracinggreen}{rgb}{0.0, 0.26, 0.15}
\definecolor{cobalt}{rgb}{0.0, 0.28, 0.67}
\DeclareSymbolFont{usualmathcal}{OMS}{cmsy}{m}{n}
\DeclareSymbolFontAlphabet{\mathcal}{usualmathcal}

\newcommand{\TT}{\mathbf{T}}

\newcommand{\BC}{{\mathbb{C}}}

\newcommand{\BL}{{\mathbb{L}}}

\newcommand{\BP}{{\mathbb{P}}}
\newcommand{\BQ}{{\mathbb{Q}}}

\newcommand{\BZ}{{\mathbb{Z}}}

\newcommand{\CI}{{\mathcal I}}

\newcommand{\CL}{{\mathcal L}}

\newcommand{\CZ}{{\mathcal Z}}

\newcommand{\pt}{{\mathsf{pt}}}

\newcommand{\fix}{\mathsf{fix}}
\newcommand{\mov}{\mathsf{mov}}

\DeclareMathOperator{\Hilb}{Hilb}

\DeclareMathOperator{\red}{red}

\DeclareMathOperator{\vir}{\mathrm{vir}}

\DeclareMathOperator{\rk}{rk}

\DeclareFontFamily{OT1}{rsfs}{}
\DeclareFontShape{OT1}{rsfs}{n}{it}{<-> rsfs10}{}
\DeclareMathAlphabet{\curly}{OT1}{rsfs}{n}{it}
\renewcommand\hom{\curly H\!om}

\newcommand\Ext{\operatorname{Ext}}
\newcommand\Hom{\operatorname{Hom}}

\newcommand{\RR}{\mathbf R}

\newcommand{\DT}{\mathsf{DT}}



\usepackage[all]{xy}
\usepackage{tikz}
\usepackage{tikz-cd}
\usepackage{adjustbox}
\usepackage{rotating}
\usepackage{comment}

\usetikzlibrary{matrix,shapes,intersections,arrows,decorations.pathmorphing}
\tikzset{commutative diagrams/arrow style=math font}
\tikzset{commutative diagrams/.cd,
mysymbol/.style={start anchor=center,end anchor=center,draw=none}}

\tikzset{
shift up/.style={
to path={([yshift=#1]\tikztostart.east) -- ([yshift=#1]\tikztotarget.west) \tikztonodes}
}
}

\theoremstyle{definition}

\newtheorem*{lemma*}{Lemma}
\newtheorem*{theorem*}{Theorem}
\newtheorem*{example*}{Example}
\newtheorem*{fact*}{Fact}
\newtheorem*{notation*}{Notation}
\newtheorem*{definition*}{Definition}
\newtheorem*{prop*}{Proposition}
\newtheorem*{remark*}{Remark}
\newtheorem*{corollary*}{Corollary}

\newtheorem*{conventions*}{Conventions}

\newtheorem{definition}{Definition}[section]

\newtheorem{remark}[definition]{Remark}
\newtheorem{conjecture}[definition]{Conjecture}

\newtheoremstyle{thm} 
        {3mm}
        {3mm}
        {\slshape}
        {0mm}
        {\bfseries}
        {.}
        {1mm}
        {}
\theoremstyle{thm}
\newtheorem{theorem}[definition]{Theorem}

\newtheorem{lemma}[definition]{Lemma}
\newtheorem{prop}[definition]{Proposition}

\newtheoremstyle{ex} 
        {3mm}
        {3mm}
        {}
        {0mm}
        {\scshape}
        {.}
        {1mm}
        {}
\theoremstyle{ex}

\newtheoremstyle{sol} 
        {3mm}
        {3mm}
        {}
        {0mm}
        {\scshape}
        {.}
        {1mm}
        {}
\theoremstyle{sol}

\usepackage{amsmath,float}
\usetikzlibrary{decorations.markings}

\usepackage{amssymb}
\usepackage{enumitem}

   \DeclareMathOperator{\oO}{\mathcal{O}}
   \DeclareMathOperator{\nor}{\mathrm{nor}}